\documentclass[12pt]{amsart}
\usepackage{graphicx}
\usepackage{amssymb}
\usepackage{epstopdf}
\theoremstyle{definition}
\newtheorem{definition}{Definition}[section]
\newtheorem{example}[definition]{Example}
\newtheorem{remark}[definition]{Remark}

\theoremstyle{plain}
\newtheorem{theorem}[definition]{Theorem}
\newtheorem{lemma}[definition]{Lemma}
\newtheorem{proposition}[definition]{Proposition}
\newtheorem{corollary}[definition]{Corollary}
\numberwithin{equation}{section}

\makeatletter
\renewcommand\@makefnmark%
{\mbox{\textsuperscript{\normalfont\@thefnmark)}}}
\makeatother

\begin{document}

\title{Unions and ideals of locally strongly porous sets}


\author{Maya Altinok}
\address{Department of Mathematics, Faculty of Sciences and Arts, Mersin University, 33343 Mersin, Turkey
}
\email{mayaaltinok@mersin.edu.tr}

\author{Oleksiy Dovgoshey}
\address{Department of Function Theory, Institute of Applied Mathematics and Mechanics of NASU, 84116, Sloviansk, Ukraine}
\email{aleksdov@mail.ru}
\thanks{The second author was partially supported by grant ¹ 15 --- bb$\backslash$19 ``Metric spaces, harmonic analysis of functions and operators, singular and nonclassical problems for differential equations'' of Donetsk National University, Vinnitsa, Ukraine}

\author{Mehmet K\"{u}\c{c}\"{u}kaslan}
\address{Department of Mathematics, Faculty of Sciences and Arts, Mersin University, 33343 Mersin, Turkey}
\email{mkucukaslan@mersin.edu.tr}

\begin{abstract}
For subsets of $\mathbb R^+ = [0,\infty)$ we introduce a notion of coherently porous sets as the sets for which the upper limit in the definition of porosity at a point is attained along the same sequence. We prove that the union of two strongly porous at~$0$ sets is strongly porous if and only if these sets are coherently porous. This result leads to a characteristic property of the intersection of all maximal ideals containing in the family of strongly porous at~$0$ subsets of $\mathbb R^+$. It is also shown that the union of a set $A \subseteq \mathbb R^+$ with arbitrary strongly porous at~$0$ subset of $\mathbb R^+$ is porous at~$0$ if and only if $A$ is lower porous at~$0$.
\end{abstract}

\maketitle

\section{Introduction}

The porosity appeared in the papers of Denjoy \cite{denjoy1}, \cite{denjoy2} and Khintchine \cite{khintchine} and, independently, Dolzenko \cite{dolzenko}. This concept has found interesting applications in the theory of free boundaries \cite{karp}, generalized subharmonic functions \cite{dovgoshey}, complex dynamics \cite{przytycki}, quasisymmetric maps \cite{vaisala}, infinitesimel geometry \cite{bilet} and other areas of mathematics.
\begin{definition}[\cite{thomson}]\label{d1}
Let $E \subseteq \mathbb{R}^+$. The right upper porosity of $E$ at~$0$ is the number
\begin{equation}\label{eq1}
\overline{p} (E)= \limsup_{h \to 0^+} \frac{\lambda (E,h)}{h}
\end{equation}
where $\lambda (E,h)$ is the length of the largest open subinterval of $(0,h)$ that contains no point of $E$.
\end{definition}

The porosity of $E$ at a point $p \in \mathbb{R}^+$ has a standard interpatation as a normalized size of holes in $E$ near $p$.

We will use the following terminology. A set $E \subseteq \mathbb{R}^+$ is:
\begin{itemize}
\item Porous at~$0$ if $\overline{p}(E)> 0$;
\item Strongly porous at~$0$ if $\overline{p}(E)=1$;
\item Nonporous at~$0$ if $\overline{p}(E)= 0$.
\end{itemize}

It should be noted that the standard definitions of porous, strongly porous and nonporous sets use the bilateral porosity at a point instead of the right porosity at a point (see, for example, \cite{thomson}) but the present paper deals only with the right porosity at~$0$ of subsets of $\mathbb R^+$.

For $E \subseteq \mathbb R^+$ denote by $UMP(E)$ the set of all sequences $(h_n)_{n \in \mathbb N}$ of positive real numbers such that $\lim_{n\to\infty}h_n=0$ and
$$
\overline{p}(E) = \lim_{n\to \infty} \frac{\lambda(E,h_n)}{h_n}.
$$
We say that a pair $\{A, B\}$ of subsets of $\mathbb R^+$ is coherently porous if
$$
UMP(A) \cap UMP(B) \neq\varnothing.
$$
More generally we will use the following
\begin{definition}\label{d1.2}
Let $\mathbf{A} = \{A_i: i \in I\}$ be a family of subsets of $\mathbb R^+$. The family $\mathbf{A}$ is coherently porous if
$$
\bigcap_{i \in I} UMP(A_i) \neq \varnothing.
$$
\end{definition}

The second section of the papers contains some properties of coherently porous families of subsets of $\mathbb R^+$. In particular it is shown that the union $A\cup B$ of strongly porous at~$0$ subsets of $\mathbb R^+$ is strongly porous at~$0$ if and only if the pair $\{A, B\}$ is coherently porous, see Corollary~\ref{c2.9}. Theorem~\ref{t2.6} describes the structure of the sets $A \subseteq \mathbb R^+$ for which $\{A, C\}$ is coherently porous for every $C \subseteq \mathbb R^+$.

The third section deals with maximal ideals in the family of all strongly porous at~$0$ sets. It is shown that for any pair $\mathbf{I}_1$, $\mathbf{I}_2$ of distinct maximal ideals there is $A \subseteq \mathbb R^+$ such that $A \in \mathbf{I}_1$ and $\mathbb R^+ - A \in \mathbf{I}_2$, see Theorem~\ref{t3.13}. A set of characteristic properties of the intersection of all maximal ideals of strongly porous at~$0$ sets is given in Theorem~\ref{t3.5}.

The lower porosity at~$0$ is considered in the fourth section. We prove that $A \subseteq \mathbb R^+$ is lower porous at~$0$ if and only if $A \cup B$ is porous at~$0$ for every strongly porous at~$0$ set $B \subseteq \mathbb R^+$, see Theorem~\ref{t4.10}.

\section{Union of strongly porous at zero sets}

The present section deals with the main properties of coherently porous sets. In particular we apply this notion to describe the necessary and sufficient conditions under which the union of strongly porous at~$0$ sets is strongly porous at~$0$.

Some results about finite and countable unions of locally porous sets were obtained by Renfro in \cite{Re}.

The next example shows that the union of two strongly porous at~$0$ sets can be nonporous at~$0$.

\begin{example}
\label{ex1}
Let $(a_n)_{n \in \mathbb{N}}$ and $(b_n)_{n \in \mathbb{N}}$ be some sequences of positive real numbers such that
\begin{equation}
\label{eq10}
a_n > b_n >a_{n+1} >0
\end{equation}
for all $n \in \mathbb{N}$ and
\begin{equation}
\label{eq11}
\lim_{n \to \infty} \frac{b_n}{a_n}= \lim_{n \to \infty} \frac{a_{n+1}}{b_n}=0.
\end{equation}
It follows from \eqref{eq11} that
\begin{equation*}
\lim_{n \to \infty} a_n = \lim_{n \to \infty} b_n=0.
\end{equation*}
Let us define the sets $A,B \subseteq \mathbb{R}^+$ as
\begin{equation*}
A:=\bigcup_{n=1}^{\infty}[b_n,a_n] \quad \text{and}\quad  B:=\bigcup_{n=1}^{\infty}[a_{n+1},b_n].
\end{equation*}
It is clear that $A \cup B= (0, a_1]$. Thus $A \cup B$ is nonporous at~$0$. Using \eqref{eq10}, \eqref{eq11} and \eqref{eq1} we obtain
\begin{eqnarray*}
\overline{p}(A) &\geq& \limsup_{n \to \infty} \frac{\lambda(A,b_n)}{b_n} = \limsup_{n \to \infty} \frac{(b_n-a_{n+1})}{b_n}\\
&=&1- \liminf_{n \to \infty} \frac{a_{n+1}}{b_n}=1.
\end{eqnarray*}
Similarly we can prove that $\overline{p}(B)\geq1$. Hence, $A$ and $B$ are strongly porous at~$0$.
\end{example}

Starting from Example~\ref{ex1} it is easy to prove that every $E \subseteq \mathbb R^+$ is the union of two strongly porous at~$0$ sets.

\begin{proposition}\label{c3}
Let $E$ be a subset of $\mathbb{R}^+$. Then there are subsets $S$ and $Q$ of $\mathbb{R}^+$ such that
\begin{equation*}
E=S\cup Q
\end{equation*}
and $\overline{p}(S)=\overline{p}(Q)=1$.
\end{proposition}
\begin{proof}
Let $A$ and $B$ be subsets of $\mathbb{R}^+$ constructed in Example \ref{ex1}. Write
\begin{equation*}
S:=\big(\{ 0 \} \cup A\big)\cap E~\text{and}~Q:=(B \cup (a_1, \infty)) \cap E.
\end{equation*}
Then we have
$$
S \cup Q = \big( \{0\} \cup (0, a_1] \cup (a_1, \infty) \big) \cap E = \mathbb{R}^+ \cap E = E.
$$
Since $\overline{p}(\{0\} \cup A)=\overline{p}(A)=1$ and $\overline{p}(B \cup (a_1,\infty))=\overline{p}(B)=1$ and
$$
S \subseteq (\{0\} \cup A) \quad \text{and}\quad Q \subseteq B \cup (a_1, \infty),
$$
the equalities $\overline{p}(S)=\overline{p}(Q)=1$ hold.
\end{proof}

\begin{lemma}\label{l2.3}
Let $A$ be a porous at~$0$ subset of $\mathbb{R}^+$ and let
$$
(h_n)_{n \in \mathbb{N}} \in UMP(A).
$$
For every $n \in \mathbb{N}$, denote by $(c_n,d_n)$ the largest open interval in the set $((0,h_n)-A)$. Then the equality
\begin{equation}
\label{eq14}
\lim_{n \to \infty}\frac{d_n}{h_n}=1
\end{equation}
holds.
\end{lemma}
\begin{proof}
The inequality $h_n \geq d_n$ holds for every $n \in \mathbb{N}$. Consequently we have
\begin{equation*}
\limsup_{n \to \infty}\frac{d_n}{h_n} \leq 1.
\end{equation*}
If \eqref{eq14} does not hold, then there are $q\in \mathbb{R}^+$ and a subsequence $(h_{n_k})_{k \in \mathbb{N}}$ of the sequence $(h_n)_{n \in \mathbb{N}}$ such that
\begin{equation*}
0 < \overline{p}(A) = \lim_{n\to \infty} \frac{d_n-c_n}{h_n}\leq \liminf_{n \to \infty}\frac{d_n}{h_n} =  \lim_{k \to \infty} \frac{d_{n_k}}{h_{n_k}}=q <1.
\end{equation*}
Since
$$
(c_n,d_n) \subseteq \big((0,d_n)-A \big),
$$
we have $\lambda(A,h_n)= \lambda(A,d_n)$. Consequently
\begin{eqnarray*}
\limsup_{k \to \infty} \frac{\lambda(A,d_{n_k})}{d_{n_k}}&=&  \limsup_{k \to \infty} \frac{h_{n_k}}{d_{n_k}} \frac{\lambda(A,h_{n_k})}{h_{n_k}}\\
&=&\frac{1}{q}\lim_{k \to \infty} \frac{\lambda(A,h_{n_k})}{h_{n_k}} = \frac{1}{q} \overline{p}(A) > \overline{p}(A).
\end{eqnarray*}
The last inequality contradicts the definition of the right upper porosity at~$0$.
\end{proof}

The next lemma is straightforward and we omit its proof.
\begin{lemma}\label{l2.4}
Let $A$ be a subset of $\mathbb R$ and let
$$
(h_n)_{n \in \mathbb N} \in UMP(A).
$$
Then the following statements hold:
\begin{itemize}
\item[(\textit{i})] If $(t_n)_{n \in \mathbb N}$ is a sequence of positive real numbers such that
$$
\lim_{n\to \infty} \frac{t_n}{h_n}=1,
$$
then
$$
(t_n)_{n \in \mathbb N} \in UMP(A);
$$
\item[(\textit{ii})] Every subsequence $(h_{n_k})_{k \in \mathbb N}$ of the sequence $(h_n)_{n \in \mathbb N}$ belongs to $UMP(A)$.
\end{itemize}
\end{lemma}

We will sometimes use this lemma without any references.

\begin{lemma}\label{l2.5}
Let $A$ be a subset of $\mathbb{R}^+$ with
\begin{equation}
\label{eq23}
0< \overline{p}(A) <1
\end{equation}
and let $(h_n)_{n \in \mathbb N} \in UMP(A)$. Then there is a sequence $((q_n,t_n))_{n \in \mathbb N}$ of open intervals such that every $(q_n,t_n)$ is a connected component of the set $Int(\mathbb{R}^+-A)$ and
\begin{equation}
\label{eq25}
\lim_{n\to \infty}\frac{t_n}{h_n}=1.
\end{equation}
\end{lemma}
\begin{proof}
For every $n \in \mathbb{N}$, denote by $(c_n,d_n)$ the largest open interval in the set $(0,h_n)-A$ and, respectively, by $(q_n,t_n)$ the connected component of $Int(\mathbb R^+-A)$ which satisfies
\begin{equation}
\label{eq26}
(q_n,t_n) \supseteq (c_n,d_n).
\end{equation}
It follows from \eqref{eq23} that
\begin{equation}
\label{eq27}
\lim_{n\to \infty} t_n=0.
\end{equation}
Hence there is $n_0\in \mathbb{N}$ such that $t_n < \infty$ if $n \geq n_0$. By Lemma \ref{l2.3} the limit relation
\begin{equation}
\label{eq28}
\lim_{n\to \infty}\frac{d_n}{h_n}=1
\end{equation}
holds. Thus it suffices to show that
\begin{equation}
\label{eq29}
\lim_{n \to \infty}\frac{d_n}{t_n}=1.
\end{equation}
Inclusion \eqref{eq26} implies $d_n \leq t_n$ for every $n \in \mathbb{N}$. Consequently
\begin{equation*}
\limsup_{n \to \infty}\frac{d_n}{t_n}\leq 1.
\end{equation*}
Let $\{n_k\}_{k \in \mathbb{N}}$ be a sequence on natural numbers for which
\begin{equation}
\label{eq30}
s:= \liminf_{n \to \infty}\frac{d_n}{t_n}=\lim_{k \to \infty}\frac{d_{n_k}}{t_{n_k}}.
\end{equation}
Since $c_n=q_n$ holds for every $n\in \mathbb{N}$, we have
\begin{eqnarray*}
\frac{\lambda(A,t_{n_k})}{t_{n_k}} &=& \frac{t_{n_k}-q_{n_k}}{t_{n_k}} = \frac{t_{n_k}-d_{n_k}}{t_{n_k}} + \frac{d_{n_k}-c_{n_k}}{t_{n_k}}\\
&=& \Big( 1-\frac{d_{n_k}}{t_{n_k}} \Big)+ \frac{d_{n_k}}{t_{n_k}} \Big(\frac{d_{n_k}-c_{n_k}}{d_{n_k}}\Big)\\
&=& \Big( 1-\frac{d_{n_k}}{t_{n_k}} \Big) + \frac{d_{n_k}}{t_{n_k}} \frac{h_{n_k}}{d_{n_k}} \frac{\lambda(A,h_{n_k})}{h_{n_k}}
\end{eqnarray*}
for all sufficiently large $k\in \mathbb{N}$. Now \eqref{eq23}, \eqref{eq28} and \eqref{eq30} imply
\begin{equation*}
\overline{p}(A)\geq \limsup_{k \to \infty}\frac{\lambda(A,t_{n_k})}{t_{n_k}}=(1-s)+s\overline{p}(A).
\end{equation*}
Note that the inequalities
$$
\overline{p}(A) \geq (1-s)+s \overline{p}(A)
$$
and
\begin{equation}
\label{eq31}
\overline{p}(A)(1-s) \geq (1-s)
\end{equation}
are equivalent. From $0 < \overline{p}(A) <1$ and $1-s \geq 0$ it follows that inequality \eqref{eq31} holds if and only if $s=1$. Thus
\begin{equation*}
1= \liminf_{n\to \infty} \frac{d_n}{t_n} \leq \limsup_{n\to \infty} \frac{d_n}{t_n}\leq 1.
\end{equation*}
Limit relation \eqref{eq29} follows.
\end{proof}

\begin{remark}
The conclusion of Lemma~\ref{l2.5} is false for strongly porous at~$0$ sets $A$. It can be shown that $\overline{p}(A)=1$ if and only if $(sh_n)_{n\in \mathbb N} \in UMP(A)$ for all $(h_n)_{n\in \mathbb N} \in UMP(A)$ and $s \in (0,1)$.
\end{remark}

\begin{theorem}\label{t2.6}
Let $A$ be a subset of $\mathbb{R}^+$. Then the following statements are equivalent.
\begin{itemize}
\item[(\textit{i})] $A$ is nonporous at~$0$ or
\begin{equation*}
0 \notin \overline{A-\{0\}},
\end{equation*}
where $\overline{A-\{0\}}$ is the closure of the set $A-\{0\}$.
\item[(\textit{ii})] Every sequence $(h_n)_{n \in \mathbb N}$ of positive real numbers with $\underset{n \to \infty}{\lim}h_n=0$ belongs to $UMP(A)$.
\item[(\textit{iii})] The pair $\{A,C\}$ is coherently porous for each $C \subseteq \mathbb{R}^+$.
\end{itemize}
\end{theorem}
\begin{proof}
(\textit{i}) $\Rightarrow$ (\textit{ii}) Let (\textit{i}) hold. If $A$ is nonporous at~$0$, then for every sequence $(h_n)_{n \in \mathbb{N}}$ of positive real numbers with $\underset{n\to \infty}{\lim}h_n=0$ we have
\begin{equation*}
\overline{p}(A) \geq \limsup_{n \to \infty} \frac{\lambda(A,h_n)}{h_n} \geq \liminf_{n \to \infty} \frac{\lambda(A,h_n)}{h_n} \geq 0= \overline{p}(A).
\end{equation*}
Thus $(h_n)_{n\in \mathbb N} \in UMP(A)$.

Similarly, $0 \notin \overline{A-\{0\}}$ if and only if there is $t>0$ such that
$$
(0,t)\cap A=\varnothing.
$$
The last equality implies that $\overline{p}(A)=1$ and
\begin{equation*}
\lim_{n \to \infty} \frac{\lambda(A,h_n)}{h_n}=1
\end{equation*}
for every sequence $(h_n)_{n \in \mathbb{N}}$ of positive real numbers with $\underset{n \to \infty}{\lim}h_n=0$. The implication (\textit{i}) $\Rightarrow$ (\textit{ii}) follows.

(\textit{ii}) $\Rightarrow$ (\textit{iii}) Suppose that condition (\textit{ii}) holds. Let $C \subseteq \mathbb{R}^+$ and let $(h_n)_{n \in \mathbb{N}} \in UMP(C)$. Then $(ii)$ implies
$$
(h_n)_{n \in \mathbb{N}} \in UMP(A).
$$
Thus the pair $\{A,C\}$ is coherently porous.

(\textit{iii}) $\Rightarrow$ (\textit{i}) Let (\textit{iii}) hold. Suppose that (\textit{i}) does not hold, i.e., $0 \in \overline{A-\{0\}}$ and $\overline{p}(A) > 0$. We can find a sequence $\big((a_n,b_n)\big)_{n \in \mathbb{N}}$ of open intervals in $\mathbb{R}^+-A$ such that $(b_n)_{n \in \mathbb N} \in UMP(A)$ and
\begin{equation}\label{eq32}
\lambda(A,b_n)=b_n-a_n.
\end{equation}
\begin{equation}\label{eq33}
\lim_{n \to \infty}\frac{a_{n+1}}{b_n}=0 \text{ and } b_{n+1} <a_n < b_n
\end{equation}
for every $n \in \mathbb{N}$. The inequality $a_n < b_n$ implies that the point
\begin{equation}
\label{eq34}
c_n:=\frac{2b_na_n}{b_n+a_n}
\end{equation}
belongs to the interval $(a_n,b_n)$,
\begin{equation}
\label{eq35}
a_n <c_n <b_n.
\end{equation}
Moreover, from \eqref{eq34} it follows that
\begin{equation}
\label{eq36}
\frac{c_n-a_n}{c_n}=\frac{1}{2}\frac{b_n-a_n}{b_n}.
\end{equation}
Let us define a set $C \subseteq \mathbb{R}^+$ as
\begin{equation}
\label{eq37}
C := \mathbb{R}^+- \Big( \bigcup_{n=1}^{\infty}(a_n,c_n) \Big).
\end{equation}
It follows from \eqref{eq33} and \eqref{eq35} that
\begin{equation}
\label{eq38}
\lambda(C,c_n)=c_n-a_n.
\end{equation}
holds for all sufficiently large $n$. Since every bounded connected component of $Int(\mathbb{R}^+-C)$ has the form $(a_n,c_n)$ for some $n \in \mathbb N$, equalities \eqref{eq38}, \eqref{eq36} and \eqref{eq32} imply
\begin{eqnarray}\label{eq2.19}
\overline{p}(C)&=&\limsup_{n \to \infty}\frac{\lambda(C,c_n)}{c_n}\\
&=&\lim_{n\to \infty} \frac{c_n-a_n}{c_n}=\frac{1}{2}\lim_{n \to \infty}\frac{b_n-a_n}{b_n} =\frac{1}{2} \overline{p}(A). \nonumber
\end{eqnarray}
It follows from $0<\overline{p}(A)\leq 1$  that
\begin{equation*}
0< \overline{p}(C)\leq \frac{1}{2}.
\end{equation*}
Let $(h_k)_{k \in \mathbb{N}} \in UMP(C)$. By Lemma \ref{l2.5} there is a subsequence $(c_{n_k})_{k \in \mathbb{N}}$ of $(c_n)_{n \in \mathbb{N}}$ such that
\begin{equation*}
\lim_{k \to \infty}\frac{c_{n_k}}{h_k}=1.
\end{equation*}
By statement (\textit{iii}) the pair $\{A,C\}$ is coherently porous. Consequently
\begin{equation}\label{eq2.20}
\overline{p}(A)=\lim_{k \to \infty}\frac{\lambda(A,c_{n_k})}{c_{n_k}}.
\end{equation}
Using \eqref{eq33} we see that $(a_{n_k},c_{n_k})$ is the largest interval in the set
\begin{equation*}
(\mathbb{R}^+-A)\cap (0,c_{n_k})
\end{equation*}
for all sufficiently large $k$. Hence we have
\begin{equation*}
\lambda(A,c_{n_k})=c_{n_k}-a_{n_k}
\end{equation*}
for all sufficiently large $k$. Now from \eqref{eq2.19} and \eqref{eq2.20} we obtain
\begin{equation*}
\overline{p}(A)=\lim_{k\to \infty} \frac{c_{n_k}-a_{n_k}}{c_{n_k}}=\frac{1}{2}\overline{p}(A).
\end{equation*}
The equality
$$
\overline{p}(A)=\frac{1}{2} \overline{p}(A)
$$
is valid if and only if
\begin{equation*}
\overline{p}(A)=0.
\end{equation*}
It contradicts the condition $\overline{p}(A) > 0$.
\end{proof}

\begin{lemma}\label{l2.7}
Let $A_1,\ldots, A_n$, $n\geq 2$, be subsets of $\mathbb{R}^+$ with
$$
\overline{p}(A_i) \in (0,1]
$$
for every $i=1,\ldots,n$. If the family $\{A_1,\ldots, A_n\}$ is coherently porous, then the equality
\begin{equation}\label{eq2.21}
\overline{p}(A_1\cup \ldots\cup A_n)= \min\{ \overline{p}(A_1),\ldots, \overline{p}(A_n) \}
\end{equation}
holds.
\end{lemma}
\begin{proof}
Let $\{A_1,\ldots, A_n\}$ be coherently porous. The inclusions
$$
(A_1\cup \ldots \cup A_n) \supseteq A_i, \quad i = 1,\ldots,n,
$$
give us the inequality
$$
\overline{p}(A_1\cup \ldots \cup A_n) \leq \min\{ \overline{p}(A_1),\ldots, \overline{p}(A_n) \}.
$$
Hence it suffices to show that
\begin{equation}\label{eq2.22}
\overline{p}(A_1\cup \ldots \cup A_n) \geq \min\{ \overline{p}(A_1),\ldots, \overline{p}(A_n) \}.
\end{equation}
Since $\{A_1,\ldots, A_n\}$ is coherently porous, there exists
$$
(h_k)_{k \in \mathbb N} \in \bigcap_{i=1}^n UMP(A_i).
$$
For all $k \in \mathbb N$ and $i \in \{1,\ldots,n\}$ denote by $(c_k^i, d_k^i)$ the largest open interval in the set $(0, h_k)-A_i$. Write
\begin{equation}\label{eq2.23}
d_k:=\min\{d_k^1, \ldots, d_k^n\}.
\end{equation}
Lemma~\ref{l2.3} implies that
$$
\lim_{k\to \infty} \frac{d_k^i}{h_k} = 1
$$
for every $i \in \{1, \ldots, n\}$. Consequently we have
$$
\lim_{k\to \infty} \frac{d_k}{h_k} = 1.
$$
By Lemma~\ref{l2.4} it follows that
\begin{equation}\label{eq2.24}
(d_k)_{k \in \mathbb N} \in \bigcap_{i=1}^n UMP(A_i).
\end{equation}
Let us denote by $i_0$ an index for which
\begin{equation}\label{eq2.25}
\overline{p}(A_{i_0}) = \min\{\overline{p}(A_1),\ldots, \overline{p}(A_n)\}.
\end{equation}
Let $\overline{p}_0 := \overline{p}(A_{i_0})$ and $\varepsilon \in (1- \overline{p}_0, 1)$. Using \eqref{eq2.23}, \eqref{eq2.24} and \eqref{eq2.25} we can prove that there is $k_0 \in \mathbb N$ such that the inclusion $$
(\varepsilon d_k, d_k) \subseteq (c_k^i, d_k^i)
$$
holds for every $i \in \{1,\ldots,n\}$ if $k \geq k_0$. Consequently
$$
\overline{p}(A_1\cup \ldots \cup A_n) \geq \limsup_{k\to \infty} \frac{\lambda(A_1\cup \ldots \cup A_n, d_k)}{d_k} \geq \limsup_{k\to \infty} \frac{d_k - \varepsilon d_k}{d_k} = 1-\varepsilon
$$
for every $\varepsilon \in (1- \overline{p}_0, 1)$. Letting $\varepsilon$ to $(1-\overline{p}_0)$ we obtain \eqref{eq2.22}.
\end{proof}

\begin{theorem}\label{t4}
Let $A_1,...A_n$ be subsets of $\mathbb{R}^+$. Suppose there is a number $p \in [0,1]$ such that
\begin{equation}\label{eq22}
\overline{p}(A_1)=...=\overline{p}(A_n)=p.
\end{equation}
Then the following two statements are equivalent.
\begin{itemize}
\item[(\textit{i})] The family $\{A_1,\ldots, A_n\}$ is coherently porous.
\item[(\textit{ii})] The equality
\begin{equation*}
\overline{p}(A_1\cup\ldots\cup A_n)=p
\end{equation*}
holds.
\end{itemize}
\end{theorem}
\begin{proof}
The case $p=0$ directly follows from Theorem~\ref{t2.6}.

Let $p \in (0, 1]$ and let $\{A_1, \ldots, A_n\}$ be coherently porous. Then by Lemma~\ref{l2.7} we have
$$
\overline{p}(A_1\cup\ldots \cup A_n) = \min \{\overline{p}(A_1), \ldots, \overline{p}(A_n)\} = p.
$$
The implication $(i) \Rightarrow (ii)$ holds.

Let $\overline{p}(A_1\cup\ldots \cup A_n) = p$. To prove $(ii) \Rightarrow (i)$ it suffices to show that
\begin{equation}\label{eq21}
(h_k)_{k \in \mathbb N} \in \bigcap_{i=1}^n UMP(A_i)
\end{equation}
holds if
\begin{equation}\label{eq20}
(h_k)_{k \in \mathbb N} \in UMP(A_1\cup\ldots \cup A_n).
\end{equation}
Suppose $(h_k)_{k \in \mathbb N}$ satisfies \eqref{eq20}. Then we have
$$
p = \lim_{k\to \infty}\frac{\lambda(A_1\cup\ldots \cup A_n, h_k)}{h_k} \leq \liminf_{k\to \infty}\frac{\lambda(A_i, h_k)}{h_k}
$$
for every $i \in \{1, \ldots, n\}$. Let $1\leq i \leq n$. Since
$$
\liminf_{k\to \infty}\frac{\lambda(A_i, h_k)}{h_k} \leq \limsup_{k\to \infty}\frac{\lambda(A_i, h_k)}{h_k} \leq p,
$$
there exists $\lim_{k\to \infty}\frac{\lambda(A_i, h_k)}{h_k}$ such that
$$
\lim_{k\to \infty}\frac{\lambda(A_i, h_k)}{h_k} = p = \overline{p}(A_i).
$$
Thus the statement
$$
(h_k)_{k \in \mathbb N} \in UMP(A_i)
$$
holds for every $i \in \{1, \ldots, n\}$. Statement \eqref{eq21} follows.
\end{proof}

\begin{corollary}\label{c2.9}
Let $A_1,...,A_n$ be strongly porous at~$0$ subsets of $\mathbb{R}^+$. Then the union $\cup_{j=1}^{n} A_j$ is strongly porous at~$0$ if and only if the family $\{A_1, \ldots, A_n\}$ is coherently porous.
\end{corollary}

In the case of coherently porous $\mathbf{A} = \{A_n: n \in \mathbb N\}$ the union $\cup_{n\in \mathbb N} A_n$ can be nonporous even if every $A_n$ is strongly porous at~$0$.

\begin{proposition}\label{p2.10}
For every $A \subseteq \mathbb R^+$ there is a countable coherently porous family $\mathbf{A} = \{A_n: n \in \mathbb N\}$ of strongly porous at~$0$ sets such that
\begin{equation}\label{p2.10eq1}
A = \bigcup_{n\in \mathbb N} A_n.
\end{equation}
\end{proposition}
\begin{proof}
For $A \subseteq \mathbb R^+$ and $n \in \mathbb N$ define the set $A_n$ as
$$
A_n:= A \cap \left(\{0\} \cup \left(\frac{1}{n}, \infty\right)\right).
$$
Then \eqref{p2.10eq1} is evident and $\mathbf{A} = \{A_n: n \in \mathbb N\}$ is coherently porous by Theorem~\ref{t2.6}.
\end{proof}

The next proposition collects together some basic properties of the binary relation ``be coherently porous".

\begin{proposition}\label{p2.11}
Let $A$ and $B$ be subsets of $\mathbb{R}^+$.
\begin{itemize}
\item[(\textit{i})] The pair $\{A,B\}$ is coherently porous if and only if the pair $\{B,A\}$ is coherently porous.
\item[(\textit{ii})] The pair $\{A,A\}$ is coherently porous.
\item[(\textit{iii})] If $A$ strongly porous at~$0$, then the family $2^A$ of all subsets of $A$ is coherently porous.
\item[(\textit{iv})] If $A$ and $B$ are strongly porous at~$0$, then there exists $C \subseteq \mathbb R^+$ such that $C$ is strongly porous at~$0$ and the pairs $\{A,C\}$ and $\{C,B\}$ are coherently porous.
\end{itemize}
\end{proposition}
\begin{proof}
Statements $(i)$ and $(ii)$ follow directly from Definition~\ref{d1.2}.

(\textit{iii}) The proof of this statement is similar to the proof of implication $(ii) \Rightarrow (i)$ in Theorem~\ref{t4}.

(\textit{iv}) Let $A$ and $B$ be strongly porous at~$0$. Then we have
\begin{equation*}
A \cap B \subseteq A~\text{and}~ A\cap B \subseteq B.
\end{equation*}
Using statement (\textit{iii}) we obtain that $\big\{A,(A \cap B)\big\}$ and $\big\{B,(A \cap B)\big\}$ are coherently porous. Statement (\textit{iv}) follows.
\end{proof}

In the rest of this section we describe an interesting link between strongly porous at~$0$ sets and the graph theory.

Write $A \asymp B$ if the pair $\{A,B\}$ is coherently porous and $A\neq B$. Let us denote by $\mathcal{SP}$ the set of all strongly porous at~$0$ subsets of $\mathbb{R}^+$.

Let us define a graph $G = (V, E)$ with the vertex set $V = \mathcal{SP}$ and the edge set $E$ such that vertices $A$, $B \in V$ are adjacent if and only if $A \asymp B$. Then statement $(iv)$ of Proposition~\ref{p2.11} means that the diameter of $G$ is $2$, i.e., for every two $X$, $Y \in V$ there is $Z \in V$ such that $X$, $Z$ are adjacent and $Z$, $Y$ are also adjacent.

The graphs of diameter $2$ have nice combinatorial properties which can be reformulated in the language of porosity of sets. An example of such a reformulation is given below.

\begin{proposition}\label{p2.12}
Let $\mathbf{A} = \{A_1, \ldots, A_n\}$ be a family of strongly porous at~$0$ sets and let $n = d^2$ for some integer number $d \geq 2$. Suppose that for every pair of distinct $A_i$, $A_j \in \mathbf{A}$ there is $A_k \in \mathbf{A}$ such that $A_i \asymp A_k$ and $A_k \asymp A_j$, for every $A_i \in \mathbf{A}$, the number of $A_j \in \mathbf{A}$ with $A_i \asymp A_j$ is at most $d$, and there is $A_{i_0}$ such that the number of $A_i \in \mathbf{A}$ with $A_{i_0} \asymp A_i$ is $d$. Then $d=4$ and the elements of $\mathbf{A}$ can be numerated such that
\begin{enumerate}
\item[$(i)$] $(A_i \asymp A_j)$ holds if and only if $i=j \pmod 2$.
\end{enumerate}
\end{proposition}

The corresponding result for arbitrary finite graph of diameter $2$ is proved in~\cite{EFH}.

\begin{example}
Let $A$ and $B$ be the sets from Example~\ref{ex1} and let $a \in A$ and $b \in B$. Write
$$
A_1:=A, \quad A_2:=B, \quad A_3:=A - \{a\} \quad\text{and}\quad A_4:=B - \{b\}.
$$
Then the family
$$
\mathbf{A} :=\{A_1, A_2, A_3, A_4\}
$$
yields condition~$(i)$ from Proposition~\ref{p2.12}.
\end{example}

\section{Maximal ideals in $\mathcal{SP}$.}

As in the preceding section, $\mathcal{SP}$ denotes the set of all strongly porous at~$0$ subsets of $\mathbb R^+$. The main goal of the section is to describe the set of subsets $E \subseteq \mathbb R^+$ for which $\{E, A\}$ is coherently porous for every strongly porous at~$0$ set $A$ on maximal ideals language. Moreover, it is shown that for every two distinct maximal ideals $\mathbf{I}_1$, $\mathbf{I}_2 \subseteq \mathcal{SP}$ there are $A_1 \in \mathbf{I}_1$ and $A_2 \in \mathbf{I}_2$ such that $A_1 \cup A_2 = \mathbb R^+$.

\begin{definition}\label{d3.1}
A nonempty collection $\mathbf{I}$ of subsets of a set $X$ is an ideal on $X$ if the following conditions are valid:
\begin{enumerate}
\item[($i$)] The implication
\begin{equation}\label{eq3.1}
(B \in \mathbf{I}\ \&\ C \subseteq B) \Rightarrow (C \in \mathbf{I})
\end{equation}
holds for all sets $C$ and $B$;
\item[($ii$)] $B \cup C \in \mathbf{I}$ for all $B$, $C \in \mathbf{I}$;
\item[($iii$)] $X \notin \mathbf{I}$.
\end{enumerate}
\end{definition}

A collection $\mathbf{I} \subseteq 2^X$ is said to be closed under subsets if statement $(i)$ of Definition~\ref{d3.1} is valid. It is clear that the set $\mathcal{SP}$ is closed under subsets but, as Example~\ref{ex1} shows, it is not an ideal on $\mathbb R^+$.

\begin{definition}\label{d3.2}
Let $\mathbf{\Gamma} \subseteq 2^X$ be nonempty and closed under subsets. An ideal $\mathbf{I}$ on X is $\mathbf{\Gamma}$--maximal if $\mathbf{I} \subseteq \mathbf{\Gamma}$ and the implication
\begin{equation}\label{eq3.2}
(\mathbf{I} \subseteq \mathbf{J} \subseteq \mathbf{\Gamma}) \Rightarrow (\mathbf{I} = \mathbf{J})
\end{equation}
holds for every ideal $\mathbf{J}$ on $X$.
\end{definition}

It is clear that the intersection of any nonempty family of ideals on $X$ is also an ideal on $X$. Write $M(\mathbf{\Gamma})$ for the set of all $\mathbf{\Gamma}$--maximal ideals and define an ideal $\hat{I}(\mathbf{\Gamma})$ as
\begin{equation}\label{eq3.3}
\hat{I}(\mathbf{\Gamma}) := \bigcap_{\mathbf{I} \in M(\mathbf{\Gamma})} \mathbf{I}.
\end{equation}

In what follows we describe some properties of the ideal $\hat{I}(\mathcal{SP})$ which was introduced in~\cite{BDP}. The following lemma is a particular case of Theorem~4.4 from~\cite{BDP}.

\begin{lemma}\label{l3.3}
Let $\mathbf{\Gamma} \subseteq 2^{\mathbb R^+}$ be closed under subsets, let
$$
\mathbb R^+= \bigcup_{A \in \mathbf{\Gamma}} A
$$
and let $B \subseteq \mathbb{R}^+$. Then $B \in \hat{I}(\mathbf{\Gamma})$ if and only if $B \cup E \in \mathbf{\Gamma}$ holds for every $E \in \mathbf{\Gamma}$.
\end{lemma}

\begin{corollary}\label{c3.3}
Let $A$ be a subset of $\mathbb R^+$. Then $A \in \hat{I}(\mathcal{SP})$ holds if and only if
$$
A \cup E \in \mathcal{SP}
$$
for every $E \in \mathcal{SP}$.
\end{corollary}

For $q>1$ and $E \subseteq \mathbb R^+$ the $q$--blow up of $E$ is the set
\begin{equation}\label{e3.4}
E(q):=\bigcup_{x \in E} (q^{-1}x, qx).
\end{equation}
The set $E(q)$ is open for every $E \subseteq \mathbb R^+$ and $q>1$. If the set $Cc^1(E(q))$ of connected components of $(0,1) \cap E(q)$ is infinite, then there is a unique sequence $((a_i, b_i))_{i \in \mathbb N}$ of open intervals such that:
\begin{enumerate}
\item[$(i_1)$] For every $(a,b) \in Cc^1(E(q))$ there is a unique $i_0 \in \mathbb N$ satisfying the equality
$$
(a,b) = (a_{i_0}, b_{i_0});
$$
\item[$(i_2)$] The inequalities
$$
b_{i+1} < a_i < b_{i}
$$
hold for every $i \in \mathbb N$;
\item[$(i_3)$] The equality
$$
(0, 1)\cap E(q) = \bigcup_{i=1}^{\infty} (a_i, b_i)
$$
holds.
\end{enumerate}

If a sequence $((a_i, b_i))_{i \in \mathbb N}$ satisfies $(i_1)-(i_3)$ we write
\begin{equation*}
Cc^1(E(q)) = ((a_i, b_i))_{i \in \mathbb N}.
\end{equation*}
The next lemma is a reformulation of Theorem~6.6 from~\cite{BDP}.

\begin{lemma}\label{l3.4}
The following conditions are equivalent for every $E \subseteq \mathbb R$ with $0 \in \overline{E-\{0\}}$.
\begin{enumerate}
\item[$(i)$] $E \in \hat{I}(\mathcal{SP})$.
\item[$(ii)$] For every $q>1$ there is an infinite sequence $((a_i, b_i))_{i \in \mathbb N}$ such that
\begin{equation*}
Cc^1(E(q)) = ((a_i, b_i))_{i \in \mathbb N} \text{ and }\limsup_{i\to\infty}\frac{b_i}{a_{i+1}} < \infty.
\end{equation*}
\end{enumerate}
\end{lemma}

Corollary~\ref{c3.3}, Lemma~\ref{l3.4} and Corollary~\ref{c2.9} give us the following theorem.

\begin{theorem}\label{t3.5}
The following statements are equivalent for every $E \in \mathcal{SP}$.
\begin{enumerate}
\item[$(i)$] $E \in \hat{I}(\mathcal{SP})$.
\item[$(ii)$] The pair $\{E, A\}$ is coherently porous for every $A \in \mathcal{SP}$.
\item[$(iii)$] The set $E \cup A$ is strongly porous at~$0$ for every $A \in \mathcal{SP}$.
\item[$(iv)$] We have either $0 \notin \overline{E-\{0\}}$ or, for every $q>1$, there is an infinite sequence $((a_i, b_i))_{i \in \mathbb N}$ of open interval such that
\begin{equation*}
Cc^1(E(q)) = ((a_i, b_i))_{i \in \mathbb N} \text{ and }\limsup_{i\to\infty}\frac{b_i}{a_{i+1}} < \infty.
\end{equation*}
\end{enumerate}
\end{theorem}

In the next example we construct a set $E \in \hat{I}(\mathcal{SP})$ for which condition $(i)$ from Theorem~\ref{t2.6} does not hold.

\begin{example}\label{ex3.6}
Let $(x_n)_{n \in \mathbb N}$ be a sequence of positive real numbers such that
\begin{equation}\label{eq3.4}
\lim_{n\to \infty} \frac{x_{n+1}}{x_n} = 0
\end{equation}
and let $q> 1$. Write
\begin{equation}\label{eq3.5}
E :=\bigcup_{n=1}^{\infty} (q^{-1}x_n, qx_n).
\end{equation}
Then the pairs $\{E, A\}$ are coherently porous for all strongly porous at~$0$ sets $A \subseteq \mathbb R^+$. The last statement follows from Theorem~\ref{t3.5}, condition~\eqref{eq3.4}, and the equality
\begin{equation}\label{ex3.6e1}
X(q)(q_1) = X(qq_1)
\end{equation}
where $X$ is the set of elements of the sequence $(x_n)_{n \in \mathbb N}$ and $X(q)(q_1)$ is the $q_1$--blow up of the $q$--blow of the set $X$.
\end{example}

\begin{remark}\label{r3.7}
The set $E$ from Example~\ref{ex3.6} belongs to the so-called completely strongly porous at~$0$ sets which form a proper subset of $\hat{I}(\mathcal{SP})$. For details see~\cite{BD} and~\cite{BDP}.
\end{remark}

For every nonempty set $X$, each $2^X$--maximal ideal $\mathbf{I}$ generates the ultrafilter $\mathcal{F}_\mathbf{I}$ on $X$ as
\begin{equation}\label{eq3.6}
\mathcal{F}_\mathbf{I} := \{X-E: E \in \mathbf{I}\}
\end{equation}
and conversely, for every ultrafilter $\mathcal{F}$ on $X$ the set
$$
\mathbf{I}_\mathcal{F} = \{X-E: E \in \mathcal{F}\}
$$
is a $2^X$--maximal ideal. It is well known that a filter $\mathcal{F}$ on $X$ is an ultrafilter if and only if, for every $A \subseteq X$, either $A \in \mathcal{F}$ or $(X-A)\in \mathcal{F}$. It implies

\begin{proposition}\label{p3.8}
Let $X$ be a nonempty set. Then for every two distinct $2^X$--maximal ideals $\mathbf{I}_1$ and $\mathbf{I}_2$ there is a set $A \subseteq X$ such that $A \in \mathbf{I}_1$ and $X-A \in \mathbf{I}_2$.
\end{proposition}
\begin{proof}
Since $\mathbf{I}_1 \neq \mathbf{I}_2$, there is $A \subseteq X$ such that
$$
A \in \mathbf{I}_1 \text{ and } A \notin \mathbf{I}_2.
$$
By~\eqref{eq3.6}, the statement $(X-A) \notin \mathcal{F}_{\mathbf{I}_2}$ holds. Using the above mentioned characteristic property of ultrafilters we see that $A \in \mathcal{F}_{\mathbf{I}_2}$. The last statement is equivalent to $(X-A) \in \mathbf{I}_2$.
\end{proof}

Now we will show that a result similar to Proposition~\ref{p3.8} is valid for $\mathcal{SP}$--maximal ideals.

The following two lemmas are simple modifications of the corresponding results from~\cite{BDP}.

\begin{lemma}\label{l3.9}
Let $E \subseteq \mathbb R^+$ and $E \notin \mathcal{SP}$. Then there is $q>1$ such that the equality
\begin{equation*}
E(q) \cap (0, 1) = (0, 1)
\end{equation*}
holds.
\end{lemma}
\begin{proof}
Using~\eqref{ex3.6e1} we see that the lemma is valid if
\begin{equation}\label{eq3.7}
E(q) \cap (0, t) = (0, t)
\end{equation}
holds for some $t>0$. The last equality is evident for every $q>1$ if there is $t>0$ such that $(0, t) \subseteq E$. Hence we can assume $(0, t)\setminus E \neq \varnothing$ for every $t>0$. Since $E$ is not strongly porous at~$0$, there is $s \in (0,1)$ such that
$$
\limsup_{h\to 0+}\frac{\lambda(E,h)}{h} < s.
$$
Consequently, there exists $t>0$ such that, for every $y\in (0,t) \setminus E$, we can find $x \in E$ satisfying the inequalities
$$
x < y \quad \text{and}\quad \frac{y-x}{y} < s.
$$
These inequalities imply
$$
x < y < \frac{x}{1-s}.
$$
Hence, $y \in (q^{-1}x, qx)$ holds with $q=1/(1-s)$. Thus, the inclusion $(0,t)\setminus E \subseteq E(q)$ holds for such $q$. Since $E \cap (0,t) \subseteq E(q)$ for all $t>0$ and $q>1$, we obtain
$$
(0,t) = (E \cap (0,t)) \cup ((0,t)\setminus E) \subseteq E(q) \cup E(q) = E(q).
$$
Thus
$$
(0,t) \subseteq (0, t) \cap E(q) \subseteq (0,t)
$$
which implies~\eqref{eq3.7}.
\end{proof}

\begin{lemma}\label{l3.11}
Let $E \subseteq \mathbb R^+$ and $q>1$. Then $E \in \mathcal{SP}$ if and only if $E(q) \in \mathcal{SP}$.
\end{lemma}
\begin{proof}
Since $E(q) \supseteq (E \setminus\{0\})(q)$ and
$$
(E \in \mathcal{SP}) \Leftrightarrow (E \setminus\{0\}\in \mathcal{SP}),
$$
the implication $(E(q) \in \mathcal{SP}) \Rightarrow  (E\in \mathcal{SP})$ is trivial.

Let $E\in \mathcal{SP}$. Then there is a sequence $((a_n, b_n))_{n \in \mathbb N}$ such that
$$
0<a_n < b_n, \quad \lim_{n\to \infty} b_n=0, \quad (a_n, b_n)\cap E = \varnothing \quad \text{and}\quad \lim_{n\to \infty} \frac{a_n}{b_n} = 0.
$$
It is easy to prove that
$$
qa_n < q^{-1}b_n\quad \text{and}\quad (qa_n, q^{-1}b_n) \cap E(q)=\varnothing
$$
for all sufficiently large $n$. Since
$$
\lim_{n\to \infty} \frac{qa_n}{q^{-1}b_n} = q^2\lim_{n\to \infty} \frac{a_n}{b_n} = 0,
$$
the set $E(q)$ is strongly porous at~$0$. The implication
$$
(E\in \mathcal{SP})\Rightarrow (E(q) \in \mathcal{SP})
$$
follows.
\end{proof}

\begin{lemma}\label{l3.12}
Let $\mathbf{I}$ be a $\mathcal{SP}$--maximal ideal and let $q>1$. Then the statement
\begin{equation}\label{l3.12e1}
E(q) \in \mathbf{I}
\end{equation}
holds for every $E \in \mathbf{I}$.
\end{lemma}
\begin{proof}
Let $E \in \mathbf{I}$. If $E(q) \cup A \in \mathcal{SP}$ holds for every $A \in \mathbf{I}$, then the set $\mathbf{J}$ defined by the rule
$$
(X \in \mathbf{J}) \Leftrightarrow (\exists B \in \mathbf{I} \text{ such that } X \subseteq B \cup E(q))
$$
is an ideal on $\mathbb R^{+}$ for which
\begin{equation}\label{l3.12e2}
\mathbf{I} \subseteq \mathbf{J} \subseteq \mathcal{SP}.
\end{equation}
Since $\mathbf{I}$ is $\mathcal{SP}$--maximal ideal, \eqref{l3.12e2} implies the equality $\mathbf{I} = \mathbf{J}$. It is clear that $E(q) \in  \mathbf{J}$. Thus, if $E(q)\notin \mathbf{I}$, then there is $A \in \mathbf{I}$ such that
\begin{equation}\label{l3.12e3}
E(q) \cup A \notin \mathcal{SP}.
\end{equation}
From $E \in \mathbf{I}$ and $A \in \mathbf{I}$ it follows that $A \cup E \in \mathbf{I}$. In particular, $A \cup E \in \mathbf{I}$ implies $A \cup E \in \mathcal{SP}$. Using Lemma~\ref{l3.11} we obtain
$$
(A \cup E)(q) \in \mathcal{SP},
$$
where $(A \cup E)(q)$ is the $q$--blow up of the set $A \cup E$. It follows directly from~\eqref{e3.4} that
$$
(A \cup E)(q) = A(q) \cup E(q).
$$
Moreover we evidently have the inclusion
$$
A \subseteq A(q) \cup \{0\}.
$$
Hence
$$
E(q) \cup A \subseteq E(q) \cup A(q) \cup \{0\} \in \mathcal{SP},
$$
which contradicts~\eqref{l3.12e3}.
\end{proof}

\begin{theorem}\label{t3.13}
Let $\mathbf{I}_1$ and $\mathbf{I}_2$ be two distinct $\mathcal{SP}$--maximal ideals. Then there is $A \in \mathbf{I}_1$ such that $(\mathbb R^+ - A) \in \mathbf{I}_2$.
\end{theorem}
\begin{proof}
Write
$$
\mathbf{I}_1 \vee \mathbf{I}_2:= \{A_1 \cup A_2: A_1 \in \mathbf{I}_1,\  A_2 \in \mathbf{I}_2\}.
$$
It is easy to see that
\begin{equation}\label{t3.13e1}
\mathbf{I}_1 \subseteq \mathbf{I}_1 \vee \mathbf{I}_2 \text{ and } \mathbf{I}_2 \subseteq \mathbf{I}_1 \vee \mathbf{I}_2.
\end{equation}
If $\mathbf{I}_1 \vee \mathbf{I}_2 \subseteq \mathcal{SP}$, then $\mathbf{I}_1 \vee \mathbf{I}_2$ is an ideal and from~\eqref{t3.13e1} it follow that
\begin{equation}\label{t3.13e2}
\mathbf{I}_1 = \mathbf{I}_1 \vee \mathbf{I}_2 = \mathbf{I}_2
\end{equation}
because $\mathbf{I}_1$ and $\mathbf{I}_2$ are $\mathcal{SP}$--maximal ideals. The supposition $\mathbf{I}_1 \neq \mathbf{I}_2$ contradicts~\eqref{t3.13e2}. Consequently we can find $A_1 \in \mathbf{I}_1$ and $A_2 \in \mathbf{I}_2$ such that
$$
A_1 \cup A_2 \notin \mathcal{SP}.
$$
Now using Lemma~\ref{l3.9} we can find $q>1$ for which
\begin{equation}\label{t3.13e3}
(A_1 \cup A_2)(q) \supseteq (0, 1).
\end{equation}
Write
\begin{equation}\label{t3.13e4}
A := A_1(q) \cup \{0\} \cup [1, \infty).
\end{equation}
By Lemma~\ref{l3.12} we have $A_1(q) \in \mathbf{I}_1$. Since $\{0\} \in \mathbf{I}_1$ and $[t, \infty) \in \mathbf{I}_1$, we also have $A \in \mathbf{I}_1$. Similarly $A_2(q) \in \mathbf{I}_2$ holds. From~\eqref{t3.13e3} and~\eqref{t3.13e4} we obtain that
$$
A \subseteq A_1(q) \cup A_2(q) \cup \{0\} \cup [1, \infty) = \mathbb R^{+}.
$$
Since
$$
A \cap (A_2(q)-A)=\varnothing \text{ and } A \cup (A_2(q)-A)= A \cup A_2(q) = \mathbb R^{+},
$$
the equality
$$
A_2(q)-A = \mathbb R^{+} - A
$$
holds. From $A_2(q) \in \mathbf{I}_2$ and $A_2(q)-A \subseteq A_2(q)$ it follows that
$$
\mathbb R^{+} - A \in \mathbf{I}_2.
$$
This finishes the proof.
\end{proof}

\section{The lower porosity at zero}

The notion of lower porosity at~$0$ is similar to the notion of upper porosity at~$0$.

\begin{definition}\label{d4.1}
Let $E \subseteq \mathbb R^+$. The right lower porosity of $E$ at~$0$ is the number
$$
\underline{p}(E) = \liminf_{h\to 0+} \frac{\lambda(E, h)}{h},
$$
where $\lambda(E, h)$ is the same as in Definition~\ref{d1}.
\end{definition}

We will say that $E \subseteq \mathbb R^+$ is lower porous at~$0$ if $\underline{p}(E) > 0$. The sets $E \subseteq \mathbb R^+$ with $\underline{p}(E)=0$ will be called lower nonporous at~$0$. It should be noted that $\underline{p}(E) > \frac{1}{2}$ holds if and only if $0 \notin \overline{E-\{0\}}$ (See, for example, Corollary 5.5 in~\cite{ADK}). It does not therefore make sense to introduce the concept of strongly lower porous at~$0$ sets.

The following characteristic property of lower porous at~$0$ sets is occasionally useful.

\begin{lemma}\label{l4.2}
Let $E \subseteq \mathbb R^+$. Then $E$ is lower porous at~$0$ if and only if there are $h_0>0$ and $p_0>0$ such that the equality
\begin{equation}\label{l4.2e1}
\lambda(E, h) > hp_0
\end{equation}
holds for every $h \in (0, h_0)$.
\end{lemma}
\begin{proof}
If, for every pair of positive $h_0$ and $p_0$, there is $h \in (0, h_0)$ such that
$$
\lambda(E, h) \leq hp_0,
$$
then the equality $\underline{p}(E)=0$ follows from Definition~\ref{d4.1}. Hence the equality $\underline{p}(E)>0$ implies the existence of $h_0$, $p_0$ for which~\eqref{l4.2e1} holds for every $h \in (0, h_0)$. The converse is evident.
\end{proof}

\begin{remark}\label{r4.3}
In fact for every $p_0 \in (0, \underline{p}(E))$, there is $h_0 = h_0(p_0)>0$ such that~\eqref{l4.2e1} holds for every $h \in (0, h_0)$.
\end{remark}

\begin{lemma}\label{l4.4}
Let $E \subseteq \mathbb R^+$ be lower porous at~$0$, let $p_0 \in (0, \underline{p}(E))$ and let
$$
q \in \left(1, \frac{1}{(1-p_0)^{1/2}}\right).
$$
Then the $q$-blow up $E(q)$ is also lower porous at~$0$.
\end{lemma}
\begin{proof}
By Lemma~\ref{l4.2}, it suffices to show that there are $t_0>0$ and $h_0>0$ such that
\begin{equation}\label{l4.4e1}
\lambda(E(q), h) > h t_0
\end{equation}
for all $h \in (0, h_0)$.

Since $E$ is lower porous at~$0$, there is $h_0>0$ such that
\begin{equation}\label{l4.4e2}
\lambda(E, h) > hp_0
\end{equation}
for every $h \in (0, h_0)$. (See Remark~\ref{r4.3}). Let $(a,b)$ be an interval in $((0,h) - E)$ with
$$
h \in (0, h_0) \text{ and } b-a = \lambda(E,h).
$$
Inequality~\eqref{l4.4e2} and the inequality $b \leq h$ imply
\begin{multline}\label{l4.4e3}
q^{-1}b - qa = (q^{-1}-q)b + q(b-a) = (q^{-1}-q)b + q\lambda(E,h) \\
> (q^{-1}-q)h + qhp_0 = h(q^{-1}-q(1-p_0))
\end{multline}
for every $h \in (0, h_0)$. Write
$$
t_0:=q^{-1}-q(1-p_0).
$$
Then it follows from $q \in \left(1, \frac{1}{(1-p_0)^{1/2}}\right)$ that $t_0>0$. The inequality $t_0>0$ and~\eqref{l4.4e3} show, in particular, that
$$
q^{-1}b>qa.
$$
Since the open interval $(a,b)$ lies in $((0,h) - E)$, for the open interval $(qa, q^{-1}b)$ we obtain the inclusion
$$
(qa, q^{-1}b) \subseteq ((0,h) - E(q)).
$$
Thus
$$
\lambda(E(q), h) \geq q^{-1}b-qa > h t_0
$$
holds for every $h \in (0, h_0)$.
\end{proof}

\begin{lemma}\label{l4.5}
Let $E \subseteq \mathbb R^+$ be strongly porous at~$0$. Then $(\mathbb R^+ - E)$ is lower nonporous at~$0$.
\end{lemma}
The proof is straightforward so that we omit it here.

\begin{theorem}\label{t4.6}
Let $E$ be a subset of\/ $\mathbb R^+$. Then $E$ is lower nonporous at~$0$ if and only if the set $(\mathbb R^+ - E(q))$ is strongly porous at~$0$ for every $q>1$, where $E(q)$ is the $q$-blow up of $E$.
\end{theorem}
\begin{proof}
Let $E$ be lower nonporous at~$0$. Then there is a sequence $(h_n)_{n \in \mathbb N}$ of positive real numbers such that $\lim_{n\to \infty}h_n=0$ and
\begin{equation}\label{t4.6e1}
\lim_{n\to \infty} \frac{\lambda(E, h_n)}{h_n}=0.
\end{equation}
For every $n \in \mathbb N$ denoted by $t_n$ the point closest to $h_n$ on the set $[0, h_n] \cap \overline{E}$. Equality~\eqref{t4.6e1} implies
$$
\lim_{n\to \infty} \frac{t_n}{h_n}=1.
$$
Consequently there is a sequence $(p_n)_{n \in \mathbb N}$ such that
\begin{equation}\label{t4.6e2}
p_n \in [0, h_n] \cap \overline{E}
\end{equation}
for every $n \in \mathbb N$ and
$$
\lim_{n\to \infty} \frac{p_n}{t_n}=1.
$$
It follows from~\eqref{t4.6e2} that
$$
\lambda(E, p_n) \leq \lambda(E, h_n)
$$
holds for every $n \in \mathbb N$. Hence we have
\begin{equation}\label{t4.6e3}
\lim_{n\to \infty} \frac{\lambda(E, p_n)}{p_n}=0.
\end{equation}
Let $q>1$ and $n \in \mathbb N$. Denote by $s_n$ the point of $[0, p_n]$ for which the set $(s_n, p_n]$ is a connected component of $E(q) \cap (0, p_n]$. We claim that
\begin{equation}\label{t4.6e4}
(s_n, qs_n) \cap E = \varnothing
\end{equation}
holds for every $n \in \mathbb N$. Indeed, if there is $a_n \in E$ such that
$$
a_n \in (s_n, qs_n),
$$
then by the definition of the $q$-blow up of $E$ we obtain
\begin{equation}\label{t4.6e5}
(q^{-1}a_n, qa_n) \subseteq E(q).
\end{equation}
Since
$$
q^{-1}a_n < q^{-1}qs_n = s_n,
$$
inclusion~\eqref{t4.6e5} implies
$$
(q^{-1}a_n, p_n] \subseteq E(q),
$$
contrary to the definition of $s_n$.

To prove the statement
\begin{equation}\label{t4.6e6}
(\mathbb R^+ - E(q)) \in \mathcal{SP}
\end{equation}
it suffices to show the limit relation
\begin{equation}\label{t4.6e7}
\lim_{n\to \infty} \frac{s_n}{p_n}=0.
\end{equation}
If~\eqref{t4.6e7} does not hold, then there are $c \in (0,1)$ and an increasing sequence $(n_k)_{k \in \mathbb N}$ of natural numbers such that
\begin{equation}\label{t4.6e8}
\lim_{n\to \infty} \frac{s_{n_k}}{p_{n_k}}=c.
\end{equation}
Equality~\eqref{t4.6e4} implies that
$$
\lambda(E, p_n) \geq (q-1)s_n.
$$
The last inequality and~\eqref{t4.6e8} give us
$$
\limsup_{k\to \infty} \frac{\lambda(E, p_{n_k})}{p_{n_k}} \geq (q-1)\limsup_{k\to \infty} \frac{s_{n_k}}{p_{n_k}} = c(q-1).
$$
Since $c \in (0,1)$ and $q>1$, it contradicts~\eqref{t4.6e3}. Limit relation~\eqref{t4.6e7} follows.

Let~\eqref{t4.6e6} hold for every $q>1$. We must prove that $E$ is lower porous at~$0$. Suppose that, on the contrary,
$$
\underline{p}(E)=\liminf_{h\to \infty} \frac{\lambda(E, h)}{h} > 0.
$$
By Lemma~\ref{l4.4} there is $q>1$ such that $E(q)$ is also lower porous at~$0$, i.e.~$\underline{p}(E(q))>0$. It is clear that
\begin{equation}\label{t4.6e9}
E(q) = (\mathbb R^+-(\mathbb R^+-E(q))).
\end{equation}
Since $\mathbb R^+-E(q) \in \mathcal{SP}$, Lemma~\ref{l4.5} and~\eqref{t4.6e9} imply that $E(q)$ is lower nonporous at~$0$, contrary to $\underline{p}(E(q))>0$.
\end{proof}

\begin{lemma}\label{l4.7}
Let $E \subseteq \mathbb R^+$ be porous at~$0$, let $p_0 \in (0, \overline{p}(E))$ and let
$$
q \in \left(1, \frac{1}{(1-p_0)^{1/2}}\right).
$$
Then the $q$-blow up $E(q)$ is also porous at~$0$.
\end{lemma}

The proof of this lemma can be obtained by a simple modification of the proof of Lemma~\ref{l4.4}.

To formulate the next lemma we need a generalization of the concept of $q$-blow up of sets.

Let $A \subseteq \mathbb R^+$, let $\tilde{\Gamma}=(\Gamma_n)_{n \in \mathbb N}$ be a sequence of subsets of $A$ and let $\tilde{q}=(q_n)_{n \in \mathbb N}$ be a sequence of real numbers in $(1, \infty)$.

\begin{definition}\label{d4.8}
The $(\tilde{\Gamma},\tilde{q})$-blow up of $A$ is the set
$$
A(\tilde{\Gamma},\tilde{q}):=\bigcup_{i=1}^\infty \Gamma_n(q_n),
$$
where $\Gamma_n(q_n)$ is the $q_n$-blow up of $\Gamma_n$.
\end{definition}

It is easy to set that for every $q>1$ and $A \subseteq \mathbb R^+$ we have
$$
A(q) = A(\tilde{\Gamma},\tilde{q})
$$
whenever $q_n=q$ and $\Gamma_n =A$ for every $n \in \mathbb N$.

\begin{lemma}\label{l4.9}
Let $A$ be a lower nonporous at~$0$ subset of $\mathbb R^+$. Then for every sequence $\tilde{q}=(q_n)_{n \in \mathbb N}$ with $\lim_{n\to \infty}q_n =1$ and $q_n \in (1, \infty)$ for each $n \in \mathbb N$, there is a sequence $((a_n, b_n))_{n \in \mathbb N}$ of open intervals such that:
\begin{enumerate}
\item[$(i)$] $\lim_{n\to \infty} a_n = 0$ and $0<b_{n+1} < a_n < b_n$ for every $n \in \mathbb N$;
\item[$(ii)$] The set $(\mathbb R^+ - A(\tilde{\Gamma},\tilde{q}))$ is strongly porous at~$0$ for $\tilde{\Gamma} = (\Gamma_n)_{n \in \mathbb N}$ with
$$
\Gamma_n:=(a_n, b_n) \cap A, \quad n \in \mathbb N.
$$
\end{enumerate}
\end{lemma}
\begin{proof}
Let $(q_n)_{n \in \mathbb N}$ satisfy the condition of the lemma. By Theorem~\ref{t4.6}, the set $\mathbb R^+ - A(q_1)$ is strongly porous at~$0$. Consequently there is an open interval $(a_1, b_1)$ such that
$$
a_1>0, \quad (a_1, b_1) \cap (\mathbb R^+ - A(q_1)) = \varnothing \quad\text{and}\quad \frac{b_1-a_1}{b_1}<\frac{1}{2}.
$$
The set $\mathbb R^+ - A(q_2)$ is also strongly porous at~$0$. Hence we can find $(a_2, b_2)$ such that
$$
0<a_2<b_2<a_1, \ (a_2, b_2)\cap (\mathbb R^+ - A(q_2)) = \varnothing \text{ and } \frac{b_2-a_2}{b_2}<\frac{1}{2^2}.
$$
By induction on $n$, we can find an interval $(a_n, b_n)$ such that
$$
0<a_n<b_n<a_{n-1}, \ (a_n, b_n)\cap (\mathbb R^+ - A(q_n)) = \varnothing \text{ and } \frac{b_n-a_n}{b_n}<\frac{1}{2^n}.
$$
Simple estimates show that $((a_n, b_n))_{n \in \mathbb N}$ is the desired sequence.
\end{proof}

Using the last lemma we can completely describe the sets $A \subseteq \mathbb R^+$ for which the union $A \cup B$ is porous at~$0$ for every strongly porous at~$0$ set $B \subseteq \mathbb R^+$.

\begin{theorem}\label{t4.10}
The following statements are equivalent for every subset $A$ of\/ $\mathbb R^+$.
\begin{itemize}
\item[$(i)$] The set $A$ is lower porous at~$0$.
\item[$(ii)$] The union $A\cup B$ is porous at~$0$ for every strongly porous at~$0$ set $B\subseteq \mathbb R^+$.
\end{itemize}
\end{theorem}
\begin{proof}
$(i) \Rightarrow (ii)$ Suppose that $A$ is lower porous at~$0$. Let $B$ an arbitrary strongly porous at~$0$ subset of $\mathbb R^+$ and let $(h_n)_{n \in \mathbb N} \in UMP(B)$. Using Lemma~\ref{l2.3} and Lemma~\ref{l2.4} we, without loss of generality, may assume that for every $n \in \mathbb N$ there is $s_n \in (0, h_n)$ such that
$$
(s_n, h_n) \subseteq \mathbb R^+ - B
$$
and
\begin{equation}\label{t4.10e1}
\lim_{n\to\infty} \frac{s_n}{h_n} = 0.
\end{equation}
Since $A$ is lower porous at~$0$, for every $n \in \mathbb N$ there is an interval $(t_n, p_n)$ such that
$$
(t_n, p_n) \subseteq (0, h_n) \cap (\mathbb R^+-A)
$$
and
\begin{equation}\label{t4.10e2}
\liminf_{n\to\infty} \frac{p_n-t_n}{h_n} \geq \underline{p}(A)>0.
\end{equation}
The inequality $p_n \geq p_n-t_n$ and~\eqref{t4.10e2} imply
\begin{equation}\label{t4.10e3}
\liminf_{n\to\infty} \frac{p_n}{h_n} \geq \underline{p}(A)>0.
\end{equation}
Using~\eqref{t4.10e1} we see that $p_n \in (s_n, h_n)$ for all sufficient large $n \in \mathbb N$. For every $n \in \mathbb N$ we write
$$
m_n:=\min\{t_n, s_n\}.
$$
Then, for all sufficiently large $n$, we have $m_n < h_n$ and
$$
(m_n, p_n) \subseteq (\mathbb R^+ - (A\cup B)).
$$
From~\eqref{t4.10e1} --- \eqref{t4.10e3} it follows that
\begin{multline}\label{t4.10e4}
\overline{p}(A\cup A) \geq \limsup_{n\to\infty} \frac{p_n - m_n}{h_n} \geq \liminf_{n\to\infty} \frac{p_n - m_n}{h_n} \\
\geq \min\left\{\liminf_{n\to\infty} \frac{p_n - s_n}{h_n}, \liminf_{n\to\infty} \frac{p_n - t_n}{h_n}\right\}\\
\geq \min\left\{\liminf_{n\to\infty} \frac{p_n}{h_n}, \underline{p}(A)\right\} \geq \min\{\underline{p}(A), \underline{p}(A)\} = \underline{p}(A) > 0.
\end{multline}
Thus $A \cup B$ is porous at~$0$.

$(ii) \Rightarrow (i)$ Suppose now that $A \cup B$ is porous at~$0$ for every strongly porous at~$0$ set $B \subseteq \mathbb R^+$. We must prove that $A$ is lower porous at~$0$.

Suppose, towards a contradiction, that $A$ is lower nonporous at~$0$. In this case by there are some sequences $((a_n, b_n))_{n \in \mathbb N}$ and $\tilde{q}=(q_n)_{n \in \mathbb N}$ which satisfy conditions~$(i)$, $(ii)$ and $(iii)$ from Lemma~\ref{l4.9}. In particular we have that the set $(\mathbb R^+ - A(\tilde{\Gamma}, \tilde{q}))$ is strongly porous at~$0$ with
$$
\tilde{\Gamma}= (A \cap (a_n, b_n))_{n \in \mathbb N} \text{ and } \lim_{n\to \infty} q_n = 1.
$$
Consequently the set
$$
D:=A \cup (\mathbb R^+ - A(\tilde{\Gamma}, \tilde{q}))
$$
is porous at~$0$. By Lemma~\ref{l4.7} there is $q>1$ such that the $q$-blow up $D(q)$ is also porous at~$0$. It is clear that
$$
D(q) = A(q) \cup (\mathbb R^+ - A(\tilde{\Gamma}, \tilde{q}))(q) \supseteq A(q) \cup \bigl((\mathbb R^+ - A(\tilde{\Gamma}, \tilde{q})) - \{0\}\bigr).
$$
Hence the inclusion
$$
D(q) \cap (0, t) \supseteq A(q) \cup (\mathbb R^+ - A(\tilde{\Gamma}, \tilde{q})) \cap (0,t)
$$
holds for every $t>0$. The last formula can be written as
\begin{equation}\label{t4.10e5}
D(q) \cap (0, t) \supseteq \bigl((A(q) \cap (0,t)) \cup ((0,t) - A(\tilde{\Gamma}, \tilde{q}))\bigr).
\end{equation}
Since $\lim_{n\to \infty} q_n =1$, there exists $n_0 \in \mathbb N$ such that
\begin{equation}\label{t4.10e6}
q_n < q
\end{equation}
for every $n \geq n_0$. Write
$$
t^*:= \sup \left\{x\colon x \in A \cap\bigcup_{n=1}^\infty (a_n, b_n)\right\}.
$$
Since~\eqref{t4.10e6} holds for every $n \geq n_0$, we obtain
$$
(0, t^*) \cap A(q) \supseteq (0, t^*) \cap A(\tilde{\Gamma}, \tilde{q}).
$$
The last inclusion and~\eqref{t4.10e5} with $t=t^*$ imply
$$
D(q) \cap (0, t^*) \supseteq \bigl((0,t^*)\cap (A(q)) \cup ((0,t^*) - A(\tilde{\Gamma}, \tilde{q}))\bigr) = (0, t^*).
$$
Since $D(q)$ is porous at~$0$, the open interval $(0, t^*)$ is also porous at~$0$. The last statement is obviously false.
\end{proof}

Let us denote by $\mathcal{LP}$ the set of all lower porous at~$0$ subsets of $\mathbb R^+$. Write $M(\mathcal{LP})$ for the set of all $\mathcal{LP}$-maximal ideals and define an ideal $\hat{I}(\mathcal{LP})$ as
\begin{equation}\label{eq4.20}
\hat{I}(\mathcal{LP}) = \bigcap_{\mathbf{I} \in M(\mathcal{LP})}\mathbf{I}
\end{equation}
(See formula~\eqref{eq3.3}).

\begin{corollary}\label{c4.11}
The inclusion
$$
\hat{I}(\mathcal{LP}) \supseteq\hat{I}(\mathcal{SP})
$$
holds.
\end{corollary}
\begin{proof}
Let $B \in \hat{I}(\mathcal{SP})$. We must show that $B \in \hat{I}(\mathcal{LP})$. Using Lemma~\ref{l3.3} we obtain that $B \in \hat{I}(\mathcal{LP})$ if and only if
\begin{equation}\label{c4.11e1}
B \cup E \in \mathcal{LP}
\end{equation}
for every $E \in \mathcal{LP}$. By Theorem~\ref{t4.10} statement~\eqref{c4.11e1} holds if and only if the set $(B \cup E) \cup S$ is porous at~$0$ for every $S \in \mathcal{SP}$. Corollary~\ref{c3.3} implies that
$$
B \cup S \in \mathcal{SP}
$$
for every $S \in \mathcal{SP}$. Using Theorem~\ref{t4.10} again we obtain $(i)$.
\end{proof}

Similarly to~\eqref{eq4.20} we can define the ideal
$$
\hat{I}(\mathcal{P}) = \bigcap_{\mathbf{I} \in M(\mathcal{P})} \mathbf{I},
$$
where $\mathcal{P}$ is the set of all porous at~$0$ subsets of $\mathbb R^+$ and $M(\mathcal{P})$ is the set of all $\mathcal{P}$-maximal ideals.

\begin{corollary}\label{c4.12}
The inclusion
$$
\hat{I}(\mathcal{LP}) \supseteq\hat{I}(\mathcal{P})
$$
holds.
\end{corollary}
\begin{proof}
Let $B \in \hat{I}(\mathcal{P})$. As in the proof of Corollary~\ref{c4.11} we see that $B \in \hat{I}(\mathcal{LP})$ if and only if
\begin{equation}\label{c4.12e1}
B \cup E \cup S \in \mathcal{P}
\end{equation}
holds for every $E \in \mathcal{LP}$ and every $S \in \mathcal{SP}$. Theorem~\ref{t4.10} implies that $E \cup S \in \mathcal{P}$ if $E \in \mathcal{LP}$ and $S \in \mathcal{SP}$. Using Lemma~\ref{l3.3} we obtain~\eqref{c4.12e1}.
\end{proof}

The following example shows that
\begin{equation}\label{eq4.22}
\hat{I}(\mathcal{P}) - \hat{I}(\mathcal{SP}) \neq \varnothing.
\end{equation}

\begin{example}\label{ex4.13}
Let $s \in (0,1)$ and let $A$ be the set of all elements of the sequence $(s^n)_{n\in \mathbb N}$. The set $A$ is not strongly porous at~$0$. Hence $A \notin \hat{I}(\mathcal{SP})$ holds. To see that $A \in \hat{I}(\mathcal{P})$, it suffices to show that
\begin{equation}\label{ex4.13e1}
A \cup E \in \mathcal{P}
\end{equation}
for every $E \in \mathcal{P}$. If $0 \notin \overline{E-\{0\}}$, then~\eqref{ex4.13e1} is trivial. Suppose that $0 \in \overline{E-\{0\}}$.

Let $((a_n, b_n))_{n \in \mathbb N}$ be sequence of open intervals such that
$$
b_{n+1} < a_n < b_n \text{ and } (a_n, b_n) \subseteq \mathbb R^+-E
$$
holds for every $n \in \mathbb N$ and
$$
\limsup_{n\to \infty} \frac{b_n - a_n}{b_n} = \overline{p}(E)>0.
$$
There are several cases to consider:
\begin{enumerate}
\item[$(i_1)$] there exists a subsequence $((a_{n_k}, b_{n_k}))_{k \in \mathbb N}$ of $((a_n, b_n))_{n \in \mathbb N}$ such that $A \cap (a_{n_k}, b_{n_k}) = \varnothing$ for every $k \in \mathbb N$;
\item[$(i_2)$] there is $((a_{n_k}, b_{n_k}))_{k \in \mathbb N}$ such that $\mathrm{card}(A \cap (a_{n_k}, b_{n_k}))=1$ for every $k \in \mathbb N$;
\item[$(i_3)$] there is $((a_{n_k}, b_{n_k}))_{k \in \mathbb N}$ with $\mathrm{card}(A \cap (a_{n_k}, b_{n_k})) \geq 2$ for every $k \in \mathbb N$.
\end{enumerate}
If $(i_1)$ holds, then it is clear that
$$
\overline{p}(A \cup E) = \overline{p}(E)>0.
$$
Condition $(i_2)$ implies
$$
\lambda(A \cup E, b_{n_k}) \geq \frac{1}{2} (b_{n_k}-a_{n_k})
$$
for every $k \in \mathbb N$. Hence we have
$$
\overline{p}(A \cup E) \geq \limsup_{k\to \infty} \frac{1}{2}\frac{b_{n_k}-a_{n_k}}{b_{n_k}} = \frac{1}{2}\overline{p}(E)>0.
$$
If we have $(i_3)$, then for every $k \in \mathbb N$ there is $m(k) \in \mathbb N$ such that
$$
[s^{m(k)+1}, s^{m(k)}] \subseteq (a_{n_k}, b_{n_k}).
$$
Consequently
$$
\overline{p}(A \cup E) \geq \limsup_{k\to \infty}\frac{s^{m(k)}-s^{m(k)+1}}{s^{m(k)}} = 1-s>0.
$$
Statement~\eqref{ex4.13e1} follows.
\end{example}

\begin{remark}\label{r4.14}
The set $E$ from Example~\ref{ex3.6} belongs to $\hat{I}(\mathcal{SP})$. It is easy to see that $E \notin \hat{I}(\mathcal{P})$. Indeed we evidently have $E \cup (\mathbb R^+ - E)= \mathbb R^+$. Moreover, since
$$
E =\bigcup_{n=1}^{\infty} (q^{-1}x_n, qx_n), \quad \text{with}\quad q \in (1,\infty),
$$
we have
\begin{multline*}
\overline{p}(\mathbb R^+- E) \geq \limsup_{n\to \infty} \frac{\lambda((\mathbb R^+- E), qx_n)}{qx_n} \\
\geq \limsup_{n\to \infty} \frac{qx_n - q^{-1}x_n}{qx_n} = 1-q^{-2}>0.
\end{multline*}
Thus $\hat{I}(\mathcal{SP}) - \hat{I}(\mathcal{P}) \neq \varnothing$.

It seems to be interesting to find a nonempty set $E \subseteq \mathbb R^+$ such that
$$
E \in (\hat{I}(\mathcal{LP}) - \hat{I}(\mathcal{P}))
$$
or to prove the equality
$$
\hat{I}(\mathcal{LP}) = \hat{I}(\mathcal{P}).
$$

\end{remark}

\end{document}